\documentclass[psamsfonts,11pt]{amsart}

\usepackage{amssymb,amsfonts}
\usepackage[all,arc]{xy}
\usepackage{enumerate}
\usepackage{mathrsfs}
\usepackage{url}
\usepackage{physics}
\usepackage[margin=1in]{geometry}
\usepackage{makecell}
\usepackage{tikz}
\usepackage{forest}
\usepackage{hyperref}

\newtheorem{thm}{Theorem}[section]
\newtheorem{cor}[thm]{Corollary}
\newtheorem{prop}[thm]{Proposition}
\newtheorem{lem}[thm]{Lemma}
\newtheorem{conj}[thm]{Conjecture}
\newtheorem{quest}[thm]{Question}

\theoremstyle{definition}
\newtheorem{defn}[thm]{Definition}

\theoremstyle{remark}
\newtheorem{rem}[thm]{Remark}

\makeatletter
\let\c@equation\c@thm
\makeatother
\numberwithin{equation}{section}

\title{Stanley-Wilf Limits for Patterns in Rooted Labeled Forests}

\author{Michael Ren}

\begin{document}

\begin{abstract}

Building off recent work of Garg and Peng, we continue the investigation into classical and consecutive pattern avoidance in rooted forests. We prove a forest analogue of the Stanley-Wilf conjecture for avoiding a single pattern as well as certain other sets of patterns. Our techniques are analytic, easily generalizing to different types of pattern avoidance and allowing for computations of convergent lower bounds of the forest Stanley-Wilf limit in the cases covered by our result. We end with several open questions and directions for future research, including some on the limit distributions of certain statistics of pattern-avoiding forests.

\end{abstract}

\maketitle

\section{Introduction}
\label{intro}

A sequence of distinct integers is said to \emph{avoid} a permutation, or pattern, $\pi=\pi(1)\cdots\pi(k)$ of $[k]=\{1,\ldots,k\}$ if it contains no subsequence that is in the same relative order as $\pi$. The study of pattern avoidance in permutations of $[n]$ was started in 1968 by Knuth in \cite{K}, where stack sorting was linked to permutations avoiding the pattern $231$. Since then, pattern avoidance has blossomed into a very active area of research, with many connections made to classical and contemporary results in enumerative and algebraic combinatorics \cite{Ki}. Different variants of permutation pattern avoidance, for example avoidance of consecutive patterns \cite{E}, (bi)-vincular patterns \cite{S, BCDK}, and mesh patterns \cite{BC}, have also been extensively studied, along with notions of pattern avoidance in other combinatorial structures such as binary trees \cite{R} and posets \cite{HW}.

The variant of pattern avoidance that we investigate is in rooted labeled forests, a notion introduced in 2018 by Anders and Archer in \cite{AA}. Here, we consider unordered (non-planar) rooted forests on $n$ vertices such that each vertex has a different label in $[n]$, which we call \emph{rooted forests} on $[n]$. Here, the term unordered and non-planar refers to the fact that the children of the vertices are not endowed with a linear ordering, so when a tree is drawn, its specific embedding into the plane is not relevant. Such a forest is then said to \emph{avoid} a pattern $\pi$ if the sequence of labels from the root to any leaf avoids $\pi$ in the sense described in the previous paragraph. As a special case, this includes the aforementioned case of permutation pattern avoidance when the forest is taken to be a path. Anders and Archer find the number $f_n(S)$ of forests avoiding a set $S$ of patterns in \cite{AA} for certain sets $S$. They also study \emph{forest-Wilf equivalence}, the phenomenon when $f_n(S)=f_n(S')$ for different sets $S$ and $S'$ of patterns and all $n\in\mathbb N$. Their work was continued by Garg and Peng in \cite{GP} where the authors posed several open questions. In a companion paper \cite{Re}, we resolve some of these open questions that pertain to forest-Wilf equivalences.

Our focus in this paper is on the asymptotics of the number of forests on $[n]$ avoiding certain sets $S$ of patterns. The asymptotics of permutations that classically avoid a pattern $\pi$ is governed by the \emph{Stanley-Wilf conjecture}, which states that $\lim_{n\rightarrow\infty}|\text{Av}_n(\pi)|^{1/n}$, the \emph{Stanley-Wilf limit}, exists and is finite for all patterns $\pi$, where $\text{Av}_n(\pi)$ is the number of permutations of $[n]$ that avoid $\pi$. In 1999, Arratia proved in \cite{A} that the limit exists through a supermultiplicativity argument and in 2000 Klazar proved in \cite{Kl} that the finiteness of the limit follows from the F\"uredi-Hajnal conjecture, which was proven in 2004 by Marcus and Tardos in \cite{MT}. Since then, much work has been done to study the value of various Stanley-Wilf limits. In 2013, Fox disproved in \cite{F} the widely believed conjecture that the limits are always quadratic in the pattern length, instead showing that they are generally exponential. For permutations that consecutively avoid a pattern $\pi$, a 2011 result of Ehrenborg, Kitaev, and Perry from \cite{EKP} shows using spectral theoretic methods that the proportion of permutations of $[n]$ avoiding a pattern $\pi$ is $c\lambda^n+O(r^n)$ for positive constants $c,\lambda>r$ only depending on $\pi$. It is then natural to ask how the asymptotics behave for pattern avoidance in forests, and Garg and Peng made the following forest analogue of the Stanley-Wilf conjecture in \cite{GP} with respect to classical avoidance.

\begin{conj}[{\cite[Conjecture 7.2]{GP}}]
\label{gp72}
For any set $S$ of patterns, let $f_n(S)$ and $t_n(S)$ denote the number of rooted forests and trees on $[n]$ avoiding $S$, respectively. Then \[\lim_{n\rightarrow\infty}\frac{f_n(S)^{1/n}}n\text{ and }\lim_{n\rightarrow\infty}\frac{t_n(S)^{1/n}}n\] exist and are equal.
\end{conj}

Here, a rooted tree on $[n]$ is just a connected rooted forest on $[n]$. Notably, the finiteness of the limit immediately follows from Cayley's formula: the number of rooted labeled forests and trees are $(n+1)^{n-1}$ and $n^{n-1}$, respectively, so the limit is automatically bounded above by $1$. We resolve this conjecture in the positive for a large class of sets which includes all singleton sets.

\begin{thm}
\label{fswl}
For any set $S$ of patterns in which no $\pi\in S$ begins with $1$ or in which no $\pi\in S$ begins with its largest element, \[\lim_{n\rightarrow\infty}\frac{f_n(S)^{1/n}}n\text{ and }\lim_{n\rightarrow\infty}\frac{t_n(S)^{1/n}}n\] exist and are equal.
\end{thm}

Our methods are quite different from those previously used to prove the analogous results for permutations (the existence of the Stanley-Wilf limit, which provides a growth rate for permutations avoiding patterns rather than forests avoiding patterns) and relies on analytically interpreting the relationship between $t_n$ and $f_n$, i.e. the forest structure. The proof is quite robust and immediately generalizes to sets $S$ of consecutive, (bi)vincular, or mesh patterns in which the condition in the theorem statement is satisfied, though we will restrict our focus to classical avoidance in this paper. The key use of the pattern avoidance condition is to establish the inequality $t_{n+1}\ge f_n$, after which the rest of the proof is purely analytic. For this reason, we believe that our methods may be applicable to the asymptotic enumeration of other classes of labeled forests, perhaps unrelated to pattern avoidance. Additionally, our proof allows us to compute convergent lower bounds for the forest Stanley-Wilf limits, and we do so for several patterns in Section \ref{asy} with the help of formulas and recurrences shown in \cite{AA, GP}.

The rest of the paper is organized as follows. In Section \ref{def}, we record all of the preliminary definitions that are necessary for the rest of the paper. In Section \ref{asy}, we discuss the asymptotics of pattern-avoiding forests and give the proof of Theorem \ref{fswl} along with computed lower bounds for forest Stanley-Wilf limits. We show that the limit, when it exists, is always in $\{0\}\cup[e^{-1},1]$ and classify the sets that achieve $0$ and the sets covered by Theorem \ref{fswl} that achieve $e^{-1}$. In Section \ref{fut}, we pose questions, conjectures, and potential future directions of research, including some on various limiting statistics of pattern-avoiding forests.

\section{Definitions and Notations}
\label{def}
We begin by defining all of the notions of pattern avoidance and rooted forests that we will use throughout this paper.

\begin{defn}
\label{rfdef}
A \emph{rooted labeled forest} on a set $S$ of integers is a forest on $|S|$ vertices labeled with the elements of $S$ in which every connected component has a distinguished root vertex. Each component then has the structure of an unordered rooted tree whose vertices' children do not have a specified order, and each vertex has a unique label in $S$. In a rooted forest $F$ on $S$, we let $\ell_F(v)$ denote the label of vertex $v$ and suppress the subscript if it is clear from context.
\end{defn}

\emph{Rooted labeled trees} are rooted labeled forests with one component. For the sake of brevity, we will oftentimes refer to rooted labeled trees and forests as trees and forests, respectively, and we will always specify when we refer to other types of trees or forests.

We will make use of some standard terminology for rooted trees and forests. In a rooted tree, the \emph{root} is the distinguished vertex. For each non-root vertex $v$, the \emph{parent} of $v$ is the vertex directly before $v$ in the path from the root to $v$, and every non-root vertex is a \emph{child} of its parent. Each vertex $v$ in the tree has a \emph{depth}, defined as the number of vertices on the path from the root to $v$. For example, the root has depth $1$. The \emph{depth} of a rooted tree $T$ is the maximal depth of a vertex in $T$. All of these terms naturally carry over to rooted forests. 

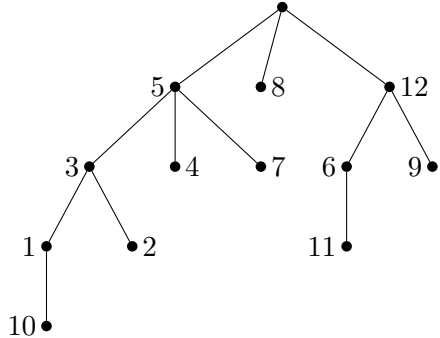
\begin{figure}[ht]
\centering
\forestset{filled circle/.style={
      circle,
      text width=4pt,
      fill,
    },}
\begin{forest}
for tree={filled circle, inner sep = 0pt, outer sep = 0 pt, s sep = 1 cm}
[, 
    [, edge label={node[left]{5}}
        [, edge label={node[left]{3}}
            [, edge label={node[left]{1}}
                [, edge label={node[left]{10}}
                ]
            ]
            [, edge label={node[right]{2}}
            ]
        ]
        [, edge label={node[right]{4}}
        ]
        [, edge label={node[right]{7}}
        ]
    ]
    [, edge label={node[right]{8}}
    ]
    [, edge label={node[right]{12}}
        [, edge label={node[left]{6}}
            [, edge label={node[left]{11}}
            ]
        ]
        [, edge label={node[left]{9}}
        ]
    ]
]
\end{forest}
\caption{A rooted labeled forest on $[12]$. The root vertex is drawn but not in the forest. We generally draw forests so that the labels of the children of every vertex are sorted in increasing order. When we draw rooted forests, we will connect the roots of each connected component to an extra unlabeled vertex and refer to this vertex as the \emph{root} of the forest. The root of the forest can be thought of as the parent of the roots of its connected components, though it is only drawn for visualization purposes and is not actually in the forest or counted when computing the depth of a vertex. In our drawings of rooted forests and trees, the root will be drawn at the top and each vertex will be drawn above its children. In this way, the path from a vertex to any of its descendants is a downward path.}
\label{rlbfig}
\end{figure}

We view rooted trees as trees with a distinguished vertex and rooted forests as a set of rooted trees. Thus, forests may be empty (have $0$ vertices), but trees will always be nonempty. An \emph{increasing} forest is a rooted labeled forest in which every vertex has a smaller label than its children so that the sequence of labels along any downward path in the forest is increasing. An increasing tree is an increasing forest with one component, and we can define decreasing forests and trees analogously.

\begin{defn}
\label{instdef}
An \emph{instance} of a pattern $\pi=\pi(1)\cdots\pi(k)$ in a rooted forest $F$ is a sequence of vertices $v_1,\ldots,v_k$ such that $v_i$ is an ancestor of $v_{i+1}$ for all $1\le i<k$ and $L(v_1),\ldots,L(v_k)$ is in the same relative order as $\pi$. We define a \emph{consecutive instance} in the same way, except we require that $v_i$ is a parent of $v_{i+1}$ instead of an ancestor so that $v_1,\ldots,v_k$ forms a downward path in the forest.
\end{defn}

\begin{defn}
\label{avodef}
A forest $F$ \emph{\emph{(}consecutively\emph{)} contains} a pattern $\pi$ if there exists a (consecutive) instance of $\pi$ in $F$, and it \emph{\emph{(}consecutively\emph{)} avoids} a set $S$ of patterns if it does not contain any (consecutive) instance of $\pi$ for any $\pi\in S$.
\end{defn}

We will oftentimes drop the braces when referring to containing or avoiding a specific singleton set $S$. The word classically may be used to describe non-consecutive avoidance or containment, so to classically avoid a set is to avoid a set as in Definition \ref{avodef}, without the parentheses. For example, the forest in Figure \ref{rlbfig} contains $213$ through the instance $5,3,10$, and it consecutively contains $312$ through the consecutive instance $12,6,11$. It avoids $123$ and consecutively avoids $213$, but it does not classically avoid $213$. A forest $F$ avoiding a set $S$ of patterns can be viewed as having the property that for every path from the root of $F$ to a leaf of $F$, the sequence of labels avoids $S$ in the sense of pattern avoidance for permutations and sequences.

For a pattern $\pi=\pi(1)\cdots\pi(k)$, we define its \emph{complement} to be the pattern $\overline\pi=k+1-\pi(1),\ldots,k+1-\pi(k)$. In other words, the complement is obtained by inverting the order of the elements. As noted in \cite[Proposition 1]{AA}, given a rooted forest $F$ on $[n]$, we may consider the rooted forest $\overline F$ defined as follows: the underlying unlabeled forest structure will be the same, but $\ell_{\overline F}(v)=n+1-\ell_F(v)$ for all vertices $v$. In other words, for all $a\in[n]$, we switch the labels $a$ and $n+1-a$. Note that any instance of a pattern $\pi$ in $F$ will become an instance of $\overline\pi$ in $\overline F$ under this relabelling, so the number of forests on $n$ vertices avoiding $\{\pi_1,\ldots,\pi_m\}$ is equal to the number of forests on $n$ vertices avoiding $\{\overline\pi_1,\ldots,\overline\pi_m\}$ for any integer $n$ and set $\{\pi_1,\ldots,\pi_m\}$ of patterns.

\section{Forest Stanley-Wilf Limits}
\label{asy}
In this section, we discuss the asymptotics of $f_n(S)$ and $t_n(S)$, the number of forests and trees, respectively, on $[n]$ that avoid a set $S$ of patterns. Our main focus will be on classical avoidance, though we will make a few remarks on how to modify our techniques and results to deal with consecutive avoidance and other types of pattern avoidance as well.

In Subsection \ref{swcpf}, we prove Theorem \ref{fswl} using analytic techniques and describe how our result can be applied to pattern avoidance in labeled forests in a very general sense. Subsection \ref{swlc} discusses the problem of determining the forest Stanley-Wilf limit of a given set $S$ of patterns, mostly those sets of patterns covered by Theorem \ref{fswl}.

We will be working closely with the exponential generating functions $F_S(x)=\sum_{n\ge0}f_n(S)x^n/n!$ and $T_S(x)=\sum_{n\ge0}t_n(S)x^n/n!$ of $f_n(S)$ and $t_n(S)$. We generally only consider one set $S$ at a time, so we will often suppress $S$ from the notation and write $f_n,t_n,F(x),T(x)$ instead. We make use of many basic properties of exponential generating functions of labeled combinatorial classes, such as the fact that $F(x)=e^{T(x)}$ since a forest that avoids $S$ is a set of trees avoiding $S$. Formally, if $\mathcal F$ is the class of rooted labeled forests avoiding $S$ and $\mathcal T$ is the class of rooted labeled trees avoiding $S$, then $\mathcal F=\mathsf{SET}(\mathcal T)$. We refer readers unfamiliar with the theory of labeled combinatorial classes and analytic combinatorics to \cite{FS} for a comprehensive treatise.

Before we move on to our proofs, we record the following definition, which is made primarily for the sake of brevity in the later arguments.

\begin{defn}
\label{coverdef}
A set $S$ of patterns is \emph{covered} if it contains two patterns $\pi=\pi(1)\cdots\pi(k)$ and $\sigma=\sigma(1)\cdots\sigma(\ell)$ with $\pi(1)=1$ and $\sigma(1)=\ell$. Otherwise, $S$ is said to be \emph{uncovered}.
\end{defn}

In other words, $S$ is covered if it contains a pattern that begins with its lowest element and a pattern that begins with its highest element, which in particular means that singleton sets $S$ are uncovered (we will ignore the pattern $1$ as only the empty forest avoids it). Here, the word covered refers to the fact that the two ``ends'' of the possible values of the first element, the highest and lowest number, both appear. Note that Theorem \ref{fswl} states that the forest Stanley-Wilf limit exists for uncovered $S$. By complementation, when working with uncovered $S$ we may assume that no patterns in $S$ begin with $1$. For a covered set $S$ of patterns, there is a natural injection from forests on $[n]$ avoiding $S$ to trees on $[n+1]$ avoiding $S$ by adding a new parent of all roots in the forest whose label is either smaller or larger than all other labels. This will be a key property for showing the existence of the forest Stanley-Wilf limit.

\subsection{The forest Stanley-Wilf conjecture for uncovered sets}
\label{swcpf}
\hspace*{\fill} \\
In this subsection, we prove Theorem \ref{fswl}, which states that for any uncovered $S$ the forest Stanley-Wilf limit $\lim_{n\rightarrow\infty}f_n^{1/n}/n$ exists and is finite. We first make a few remarks before giving our proof.

In contrast with the Stanley-Wilf conjecture for permutations, the main difficulty lies not in the finiteness but in the existence of the limit. Indeed, the total number of rooted labeled forests on $[n]$ is $(n+1)^{n-1}$ by Cayley's formula, which in particular implies that $\limsup_{n\rightarrow\infty}f_n^{1/n}/n\le1$. Under the assumption that no patterns in $S$ begin with $1$, note that if $21\in S$, then the other patterns in $S$ are superfluous as they contain all $21$. In this case, $f_n=n!$ \cite{AA}, and we already know by Stirling's approximation that $\lim_{n\rightarrow\infty}(n!)^{1/n}/n=e^{-1}$. Henceforth, we will assume that no patterns in $S$ begin with $1$ and that $21\notin S$. Another consequence of this limit is that it instead suffices to show the existence of $\lim_{n\rightarrow\infty}\left(f_n/n!\right)^{1/n}=e\lim_{n\rightarrow\infty}f_n^{1/n}/n$ instead, which we do using the exponential generating function of $f_n$.

The supermultiplicativity argument given by Arratia in \cite{A} for the existence of the limit for permutations does not easily extend to forests. The analogous supermultiplicativity inequality would be $f_{m+n}/((m+n)!)\ge(f_m/m!)\cdot(f_n/n!)$, or $f_{m+n}\ge\binom{m+n}mf_mf_n$. We remark that for \emph{ordered} rooted forests, this inequality follows from the observation that one can obtain an ordered forest $F$ on $[m+n]$ avoiding $S$ by choosing an $m$-element subset $A$ of $[m+n]$ and merging an ordered forest $F_A$ on $A$ avoiding $S$ with an ordered forest $F_B$ on $B=[m+n]\setminus A$ avoiding $S$. We do so by placing the trees in $F_A$ before the trees in $F_B$ in the ordering of the trees in $F$. This does not work for unordered forests because the construction is not injective; the order of the trees in the forest no longer matters so we cannot say that every choice of $A,F_A,F_B$ results in a different forest on $[m+n]$. We were unable to repair this argument for unordered forests, despite the fact that data suggests that the inequality still holds.

We now make some definitions relevant to our proof. As before, let $T(x)=\sum_{k\ge0}t_kx^k/k!$ and $F(x)=\sum_{k\ge0}f_kx^k/k!$ be the exponential generating functions of $t_n$ and $f_n$. Let \[A(x)=T'(x)-F(x)=\frac{F'(x)}{F(x)}-F(x)=\sum_{k\ge0}\frac{(t_{k+1}-f_k)x^k}{k!},\] \[B(x)=\int_0^xA(t)dt,\] \[C(x)=e^{B(x)},\] \[D(x)=\int_0^xC(t)dt.\]

For positive integers $n$, let \[A_n(x)=\sum_{0\le k\le n}\frac{(t_{k+1}-f_k)x^k}{k!},\] \[B_n(x)=\int_0^xA_n(t)dt,\] \[C_n(x)=e^{B_n(x)},\] \[D_n(x)=\int_0^xC_n(t)dt,\] \[F_n(x)=\frac{C_n(x)}{1-D_n(x)}.\] Given nonnegative integers $0=a_0,a_1,\ldots,a_n$, where $a_k$ represents the number of objects of size $k$ in a labeled combinatorial class $\mathcal C$, let $E(a_0,\ldots,a_n)$ denote the number of objects in $\mathsf{SET}(\mathcal C)$ of size $n$. Note that $E(a_0,\ldots,a_n)/n!$ is the coefficient of $x^n$ in $\exp\left(\sum_{k=1}^na_kx^k/k!\right)$ and that $f_n=E(t_0,\ldots,t_n)$ by construction.

The rough idea in our proof is that $\limsup_{n\rightarrow\infty}\left(f_n/n!\right)^{1/n}$ is the reciprocal of the radius of convergence of $F$. To control $\liminf_{n\rightarrow\infty}\left(f_n/n!\right)^{1/n}$, we approximate $F$ from below by a series of functions $F_m$ with the property that $F_m(z)$, viewed as a function in a complex variable $z$ of a sufficiently small magnitude, admits a meromorphic continuation to the entire complex plane. By \cite[Theorem IV.10]{FS}, the corresponding limit for the coefficients of $F_m$ exist, and this limit is a lower bound for $\liminf_{n\rightarrow\infty}\left(f_n/n!\right)^{1/n}$. The theorem then follows by showing that $F_m$ tends to $F$ in an appropriate sense. The first step is to make a combinatorial observation about the coefficients of $F$ and $T$, using the condition that $S$ is uncovered. The remainder of the proof after Proposition \ref{tkfk} is essentially purely analytic.

\begin{prop}
\label{tkfk}
The inequality $t_{k+1}\ge f_k$ holds for all $k$. Thus, all of the coefficients of $A$ are nonnegative.
\end{prop}

\begin{proof}
Because none of the patterns in $S$ start with $1$, any forest on $[k]$ avoiding $S$ can be turned into a tree on $[k+1]$ avoiding $S$ by increasing all labels by $1$ and attaching the root of each tree in the forest to a new root vertex labeled $1$. This operation is clearly injective, so the proposition follows.
\end{proof}

\begin{lem}
\label{rntor}
There exist unique positive real numbers $r_n$ such that $D_n(r_n)=1$. Furthermore, the sequence $r_1,r_2,\ldots$ is nonincreasing with limit $r=\sup\{t:D(t)\le1\}$.
\end{lem}

This limit $r$ will determine the forest Stanley-Wilf limit as the radius of convergence for $F(x)$, and the $r_n$ are increasing approximations to $r$ as the radii of convergence of $F_n(x)$.

\begin{proof}
Since all of the coefficients of $A$ are nonnegative, so are all of the coefficients of $B,C,D$ as they are constructed from $A$ using integration and exponentiation. The same is true for $A_n,B_n,C_n,D_n$. By our assumption that $S$ does not contain $21$, we know that $t_2=2$. Since $f_1=1$, the coefficient of $x$ in $A(x)$ is equal to $1$. In particular, this means that $A_n,B_n,C_n,D_n$ each have a strictly positive non-constant coefficient, so they are strictly increasing functions in $x$ that tend to infinity since $A_n$ is a polynomial. We note here that $A(x),B(x),C(x),D(x)$ are defined on $x\in[0,R)$ where $R$ is their common radius of convergence. By Cayley's formula, $t_{k+1}-f_k\le t_{k+1}\le(k+1)^k$, so $R^{-1}\le\lim_{k\rightarrow\infty}\left((k+1)^k/k!\right)^{1/k}=e$ and $R$ is positive.

Since $A_n,B_n,C_n,D_n$ are finite on $[0,\infty)$, are strictly increasing, and tend to infinity, by the fact that $D_n(0)=0$ we know that the $r_n$ exist and are unique. Furthermore, since $D_1,D_2,\ldots$ is pointwise nondecreasing, $r_1,r_2,\ldots$ is nonincreasing. Note that $A$ is the pointwise increasing limit of $A_n$ on $[0,\infty)$, so $B$ is the pointwise increasing limit of $B_n$ on $[0,\infty)$ by the integral monotone convergence theorem. It then follows that $C$ is the pointwise increasing limit of $C_n$ on $[0,\infty)$, so $D$ is the pointwise increasing limit of $D_n$ on $[0,\infty)$ as well, again by the integral monotone convergence theorem. If $R=\infty$, then $D(x)$ is defined for all $x\ge0$ and tends to infinity. Hence, $D(r)=1$, $r_n$ approaches $r$ from above, and $D_n$ approaches $D$ from below pointwise.

Suppose that $R<\infty$. Note that $r\le R$. We now split into two cases, depending on whether $r<R$. If $r=R$, then it suffices to show that $r_n$ has limit $R$. Since $D_n\uparrow D$ and $R=\sup\{t:D(t)\le1\}$, $D_n(R)<1$ for all $n$. On the other hand, for all $\epsilon>0$ and sufficiently large $n$, $D_n(R+\epsilon)>1$ since $D(R+\epsilon)=\infty$. Thus, for sufficiently large $n$, $r_n<R+\epsilon$ while $r_n>R$, so $\lim_{n\rightarrow\infty}r_n=R=r$, as desired. If $r<R$, then we know that $D(r)=1$ and for some $\epsilon>0$, $D(R-\epsilon)>1$. Since $D_n\uparrow D$, for sufficiently large $n$, $D_n(R-\epsilon)>1$ and $r_n<R-\epsilon$, and by restricting to the interval $[0,R-\epsilon]$ the result is clear.
\end{proof}

\begin{prop}
\label{fnode}
The differential equation $G'(x)/G(x)-G(x)=A_n(x)$ with initial condition $G(0)=1$ has $F_n(x)$ as a unique solution.
\end{prop}

\begin{proof}
Rewriting the differential equation as $G'(x)=G(x)^2+G(x)A_n(x)$, we obtain a Bernoulli differential equation which has a unique solution with the initial condition $G(0)=1$. It is easy to verify that $F_n(0)=1$ and that $F_n$ satisfies this differential equation (in fact our construction of $F_n$ follows the solution of the Bernoulli differential equation).
\end{proof}

\begin{lem}
\label{fnmero}
For some $\epsilon_n>0$, $F_n(z)$ as a function of a complex variable $z$ is meromorphic on $\{z:|z|< r_n+\epsilon_n\}$, with its only pole in this disk at $z=r_n$.
\end{lem}

\begin{proof}
By construction, $A_n,B_n,C_n,D_n$ are all entire, and their series expansions around $z=0$ have all nonnegative real coefficients with at least one positive coefficient, so $F_n=C_n/(1-D_n)$ is meromorphic and nonconstant. Since $C_n(r_n)>0$ and $|D_n(z)|<1$ for $|z|<r_n$ due to the nonnegative real coefficients, $F_n$ has a pole at $r_n$ and no other poles in $|z|\le r_n$ by the triangle inequality. The only poles are at roots of $D_n(z)=1$, of which there are only finitely many in the compact set $|z|\le2r_n$. It follows that for some $\epsilon_n>0$, $r_n$ is the only pole in $|z|<r_n+\epsilon_n$, as desired.
\end{proof}

\begin{proof}[Proof of Theorem \ref{fswl}]
We will show that $\limsup_{k\rightarrow\infty}\left(f_k/k!\right)^{1/k}\le1/r$ and $\liminf_{k\rightarrow\infty}\left(f_k/k!\right)^{1/k}\ge 1/r_n$ for all $n$. Then, $\lim_{k\rightarrow\infty}\left(f_k/k!\right)^{1/k}=1/r$, so Stirling's approximation gives $\lim_{n\rightarrow\infty}f_n^{1/n}/n=1/(re)$.

Note that $F(x)$ solves the differential equation $G'(x)=G(x)^2+G(x)A(x)$ with initial condition $G(0)=1$, which is a Bernoulli differential equation with unique solution $C(x)/(1-D(x))$ by construction. Thus, $F(x)=C(x)/(1-D(x))$.

Now, we show that $F(x)$ converges for $x\in[0,r)$. Recall that $r\le R$, where $R$ is the common radius of convergence of $A,B,C,D$. Thus, since $F(x)=C(x)/(1-D(x))$, $F(x)$ converges as long as $x\in[0,R)$ and $D(x)<1$. By definition, $F(x)$ converges for $x\in[0,r)$. It thus follows that the radius of convergence of $F$ is at least $r$, so $\limsup_{k\rightarrow\infty}\left(f_k/k!\right)^{1/k}\le1/r$, as desired.

Finally, let $F_n(x)=\sum_{k\ge0}a_kx^k/k!$. Note that since $F_n(z)$ is meromorphic on $|z|<r_n+\epsilon_n$ with its only pole at $r_n$, $\lim_{k\rightarrow\infty}\left(a_k/k!\right)^{1/k}=r_n$ by \cite[Theorem IV.10]{FS}. Thus, it suffices to show that $f_k\ge a_k$ for all $k$. Let $T_n(x)=\log F_n(x)$, so $F_n(x)=e^{T_n(x)}$ and $T_n(0)=0$ since $F_n(0)=1$. The differential equation $F_n'(x)/F_n(x)-F_n(x)=A_n(x)$ can then be rewritten as $T_n'(x)-e^{T_n(x)}=A_n(x)$. Suppose that $T_n(x)=\sum_{k\ge0}b_kx^k/k!$. Equating coefficients in the differential equation, $b_{k+1}-E(b_0,\ldots,b_k)=t_{k+1}-f_k$ for $k\le n$ and $b_{k+1}-E(b_0,\ldots,b_k)=0$ for $k>n$. But we know that $t_{k+1}-E(t_0,\ldots,t_k)=t_{k+1}-f_k$ for all $k$ and $t_0=b_0=0$, so $t_k=b_k$ and thus $f_k=E(t_0,\ldots,t_k)=E(b_0,\ldots,b_k)=a_k$ for all $k\le n+1$ by strong induction. For $k> n+1$, we proceed using strong induction to show that $t_k\ge b_k$ and $f_k\ge a_k$, with the base case of $k\le n+1$ already shown. For the inductive step, $b_{k+1}=E(b_0,\ldots,b_k)=a_k\le f_k\le t_{k+1}$ by Proposition \ref{tkfk} and $a_{k+1}=E(b_0,\ldots,b_{k+1})\le E(t_0,\ldots,t_{k+1})=f_{k+1}$, where we are using the monotonicity of $E$ for nonnegative inputs, so we are done.

To see that $\lim_{n\rightarrow\infty}f_n^{1/n}/n=\lim_{n\rightarrow\infty}t_n^{1/n}/n$, we again make use of the inequality $t_{k+1}\ge f_k$, which gives $f_{k-1}\le t_k\le f_k$. Taking $k$th roots, dividing by $k$, and taking the limit yields $\lim_{k\rightarrow\infty}f_{k-1}^{1/k}/k\le\lim_{k\rightarrow\infty}t_k^{1/k}/k\le\lim_{k\rightarrow\infty}f_k^{1/k}/k$. But $\lim_{k\rightarrow\infty}f_{k-1}^{1/(k(k-1))}=1$ as $f_{k-1}\le(k-2)^k$. Thus, $\lim_{k\rightarrow\infty}f_{k-1}^{1/k}/k=\lim_{k\rightarrow\infty}f_{k-1}^{1/(k-1)}/(k-1)=\lim_{k\rightarrow\infty}f_k^{1/k}/k$, and the result follows.
\end{proof}

This proof allows us to compute convergent lower bounds for the Stanley-Wilf limit for $S$, which we discuss in Subsection \ref{swlc}.

\begin{rem}
\label{crem}
The condition $t_{k+1}\ge f_k$ that we used may be replaced by $t_{k+1}\ge cf_k$ for any $c>0$, and the proof is essentially the same with minor modifications. However, this condition seems to be difficult to show for covered sets $S$ even for $c<1$, and the easy proof of Proposition \ref{tkfk} does not carry over. As long as $f_n=O(t_{n+1})$, the forest Stanley-Wilf limit exists. We believe that the limit also exists when $t_{n+1}=o(f_n)$ but that there are fundamental differences between sets $S$ satisfying $f_n=O(t_{n+1})$ and sets $S$ satisfying $t_{n+1}=o(f_n)$. We will remark more on these differences in Subsection \ref{swlc} and Section \ref{fut}.
\end{rem}

\begin{rem}
\label{fncombo}
The series $F_n$ we used to approximate $F$ from below has a combinatorial interpretation. One viewpoint, essentially given in the proof of Theorem \ref{fswl}, is that we initially force equality to hold in $t_{k+1}\ge f_k$ for all $k$, and then we iteratively replace $t_n$ with its true value for all $n$ (note that $f_n$ is determined by $t_1,\ldots,t_n$). In this way, the coefficients of $F_n$ agree with the coefficients of $F$ up to $x^n$, and as $n\rightarrow\infty$, $F_n$ converges coefficientwise to $F$. However, we can also view $F_n$ as the exponential generating function of the combinatorial class $\mathscr F_n$ of forests that avoid $S$ along with the stronger condition that every vertex with more than $n$ descendants has the smallest label among all of its descendants. The asymptotics for such forests in $\mathscr F_n$ are given by $r_n$, which converge to $r$ as $n\rightarrow\infty$ by our proof. Heuristically, $\mathscr F_n$ forms a good approximation for $\mathscr F$ because in a typical forest, we expect most vertices to not have too many descendants. Furthermore, if a vertex has many descendants, than in order to avoid $S$ it is intuitively more efficient for $S$ to have a small label since the patterns in $S$ do not start with $1$, especially if $S$ contains many patterns. This relates to the rate at which $r_n$ converges to $r$, which our proof gives no insight into.
\end{rem}

\begin{rem}
\label{robust}
Throughout our proof, the pattern avoidance condition is only relevant for Proposition \ref{tkfk} to establish $t_{k+1}\ge f_k$, and after that the proof relies on the analytic interpretation of the relation between the trees and forests in a combinatorial class of rooted labeled forests. Consequently, the proof is quite robust and immediately generalizes to give forest Stanley-Wilf limits for avoiding consecutive patterns, (bi)vincular patterns, mesh patterns, any type of pattern in which the smallest element does not come first, and arbitrary combinations thereof. The limit's existence is not driven by the pattern avoidance, but rather by the tree-forest structure in the combinatorial class. Thus, we believe that our techniques may also be useful in asymptotically enumerating other types of rooted labeled forests that may be unrelated to pattern avoidance.
\end{rem}

\subsection{Determining forest Stanley-Wilf limits}
\label{swlc}
\hspace*{\fill} \\
We now turn to the problem of finding the value of the forest Stanley-Wilf limit for a given set $S$ of patterns. Much of our work in this subsection also applies to asymptotics for consecutive-, (bi)vincular-, or mesh-pattern-avoiding forests, and we leave such computations to the interested reader.

For a set $S$ of patterns, let $L_S=\lim_{n\rightarrow\infty}f_n^{1/n}/n$ denote the forest Stanley-Wilf limit for $S$. By Theorem \ref{fswl}, $L_S$ exists for all uncovered sets $S$. We will show the existence of $L_S$ for a few other sets in this subsection. We will also drop braces in the subscript in $L_S$, so for example we will write $L_{123,231}$ instead of $L_{\{123,231\}}$.

The proof of Theorem \ref{fswl} given in Subsection \ref{swcpf} allows us to compute convergent lower bounds for $L_S$. Indeed, note that $1/(er_n)$ increases to $L_S$, where $r_n$ is the unique positive root of $D_n(x)=1$, as previously defined. The functions $A_n,B_n,C_n,D_n$ are determined by $t_{k+1}$ and $f_k$ for $k\le n$, so we are able to estimate $r_n$ by computing the sequences $t_k$ and $f_k$ up to $n+1$. Anders and Archer provide many explicit formulas for $f_n$ for certain sets in \cite{AA}, and Garg and Peng give many recursions for $f_n$ for some other sets in \cite{GP}. Using these, we are able to find lower bounds for $L_S$ for certain $S$ displayed in Figure \ref{computable}.

\begin{figure}[ht]
    \centering
    \begin{tabular}{|c|c|c|c|}
    \hline
    $S$ & $n$ & Proven & Conjectured \\
    \hline
    \makecell{$123$ \\ $132$} & $350$ & $\ge0.6766$ & $\approx0.6801$ \\
    \hline
    $213$ & $2500$ & $\ge0.65493$ & $\approx0.65521$ \\
    \hline
    \makecell{$123,213$ \\ $132,213$} & $1700$ & $\ge0.555617$ & $\approx0.555843$ \\
    \hline
    $123,231$ & $800$ & $\ge0.5402$ & $\approx0.5530$ \\
    \hline
    $132,231$ & $1000$ & $\ge0.58145$ & $\le0.58421$ \\
    \hline
    $213,231$ & $2500$ & $\ge0.557725$ & $\approx0.557864$ \\
    \hline
    $123,132,213$ & $1650$ & $\ge0.51781$ & $\le0.51939$ \\
    \hline
    $123,132,231$ & $2500$ & $\ge0.53057$ & $\le0.53169$ \\
    \hline
    $132,213,231$ & $2500$ & $\ge0.48241$ & $\le0.48317$ \\
    \hline
    $123,2413,3412$ & $1800$ & $\ge0.62765$ & $\le0.62939$ \\
    \hline
    \end{tabular}
    \caption{Computed lower bounds for $L_S$.}
    \label{computable}
\end{figure}

Here, $n$ denotes the amount of terms we computed, and the lower bound in the proven column corresponds to the one found with solving $D_n(x)=1$. The conjectured column contains five conjectured values of $L_S$ given by Garg and Peng in \cite[Conjecture 7.2]{GP} and five conjectured upper bounds based on our computations. In all of the cases we computed, the sequence $f_k^{1/k}/k$ was decreasing for $k\le n$, and our five conjectured bounds correspond to the value of $f_n^{1/n}/n$. We have also added in any nontrivial forest-Wilf equivalences in the $S$ column. We did not include results of complementation in this column, but clearly those sets also have the same forest Stanley-Wilf limit.

While our proven lower bounds on $L_S$ are relatively close to the conjectured approximate values and upper bounds, in order to compute $L_S$ to arbitrary precision, one would need a method to prove convergent upper bounds on $L_S$ as well. Unfortunately, we were not able to adapt our methods from the proof of Theorem \ref{fswl} to obtain upper bounds from the first few terms of $f_n$. A natural step would be to replace the inequality $t_{k+1}\ge f_k$ with the inequality $t_k\le kf_{k-1}$. This inequality follows from the observation that a tree on $[k]$ that avoids $S$ consists of a root vertex with label $a$ and a forest on $[k-1]\setminus\{a\}$ that avoids $S$. There are $k$ choices for $a$ and for each choice of $a$, there are at most $f_{k-1}$ forests on the remaining $k-1$ vertices that work, yielding the claimed bound of $kf_{k-1}$. Note that equality holds in the inequalities $t_{k+1}\ge f_k$ and $t_k\le kf_{k-1}$ when $S=\{21\}$ and $S=\varnothing$, respectively. The method in the proof of Theorem \ref{fswl} can be viewed as starting with forests avoiding $21$, i.e. increasing forests, and iteratively adding in more forests that avoid $S$ corresponding to using higher truncations of $A(x)=\sum_{k\ge0}(t_{k+1}-f_k)x^k/k!$. We can try to take a similar approach with the upper bound, starting with all forests and iteratively removing more forests that do not avoid $S$ corresponding to higher truncations of $P(x)=\sum_{k\ge1}(kf_{k-1}-t_k)x^k/k!$. Instead of a differential equation, we get the equation $xe^{T(x)}-T(x)=P(x)$ for $T(x)$, which we can attempt to approximate with $T_n(x)$ satisfying $xe^{T_n(x)}-T_n(x)=P_n(x)=\sum_{1\le k\le n}(kf_{k-1}-t_k)x^k/k!$. Note that when $P(x)=0$, we recover the equation $xe^{T(x)}=T(x)$, the functional equation for the Cayley tree function (see \cite[Section II.5.1]{FS}). We would like for the growth rate of the coefficients of the $T_n$ that solves $xe^{T_n}-T_n=P(x)$, or $T_n=xe^{T_n}-P_n(x)$, to be in the smooth implicit-function schema defined in \cite[Section VII.4.1]{FS}, in which case we can recover an upper bound for $L_S$. However, the presence of negative coefficients in the $-P_n(x)$ on the right-hand side makes this impossible. The example given at the end of \cite[Section VII.4.1]{FS} shows that such negative coefficients can lead to pathological situations. It would be interesting to somehow repair this method or find a different way to compute upper bounds on $L_S$.

While Theorem \ref{fswl} only shows the existence of $L_S$ for uncovered $S$, it is possible to show that $L_S$ exists in other cases as well. For example, we can classify all of the sets $S$ of patterns satisfying $L_S=0$.

\begin{prop}
\label{lim0}
The limit $\lim_{n\rightarrow\infty}f_n^{1/n}/n=0$ holds if and only if $S$ contains the patterns $1\cdots k$ and $\ell\cdots 1$ for some $k$ and $\ell$.
\end{prop}

\begin{proof}
Note that forests avoiding $1\cdots k$ and $\ell\cdots1$ must have depth at most $k\ell$ by the Erd\H os-Szekeres Theorem. We will show that if $f_{m,n}$ and $t_{m,n}$ are respectively the number of forests and trees on $[n]$ of depth at most $m$, then $\lim_{n\rightarrow\infty}f_{m,n}^{1/n}/n=\lim_{n\rightarrow\infty}t_{m,n}^{1/n}/n=0$. Let $F_m(x)=\sum_{k=0}^\infty f_{m,k}x^k/k!$ and $T_m(x)=\sum_{k=0}^\infty t_{m,k}x^k/k!$ denote the exponential generating functions of the sequences $\{f_{m,n}\}$ and $\{t_{m,n}\}$. By standard manipulations of labeled combinatorial classes and exponential generating functions, $F_m=e^{T_m}$ and $T_{m+1}=xF_m$ for all $m$. As $T_1=x$, it follows by induction on $m$ that $F_m(z)$ and $T_m(z)$ are entire functions in $z\in\mathbb C$ for all $m$. Thus, $\lim_{n\rightarrow\infty}\left(f_{m,n}/n!\right)^{1/n}=\lim_{n\rightarrow\infty}\left(t_{m,n}/n!\right)^{1/n}=0$, so by Stirling's approximation, \[0\le\lim_{n\rightarrow\infty}f_n^{1/n}/n=e\lim_{n\rightarrow\infty}\left(f_n/n!\right)^{1/n}\le e\lim_{n\rightarrow\infty}\left(f_{m,n}/n!\right)^{1/n}=0.\]

On the other hand, note that if all increasing forests avoid $S$, then $f_n\ge n!$ so $\liminf_{n\rightarrow\infty}f_n^{1/n}/n\ge e^{-1}$. Thus, $S$ must contain a pattern of the form $1\cdots k$. The same holds for decreasing forests, so $S$ must also contain a pattern of the form $\ell\cdots1$, as desired.
\end{proof}

\begin{cor}
\label{limint}
When $L_S$ exists, it lies in $\{0\}\cup[e^{-1},1]$.
\end{cor}

By examining our proof of Theorem \ref{fswl}, we can also determine when an uncovered set $S$ satisfies $L_S=e^{-1}$.

\begin{prop}
\label{lim1e}
If $S$ is an uncovered set of patterns, then $L_S=e^{-1}$ if and only if $S$ contains $12$ or $21$.
\end{prop}

\begin{proof}
If $21\in S$, then no patterns in $S$ can start with $1$, so all other patterns contain $21$ and are superfluous. It then follows that $f_n=n!$ so $L_S=e^{-1}$. The same exact argument works for if $12\in S$. In the other direction, note that $1/(er_1)$ is a lower bound for $L_S$. If $S$ does not contain $12$ or $21$, then $t_2=2$ while $f_1=1$, so $A_1(x)=x$ and $D_1(x)=\int_0^xe^{t^2/2}dt$. It is then clear that $D_1(1)>1$ so $r_1<1$ and $L_S>1/e$, as desired.
\end{proof}

This proposition shows that a small change to the number of $S$-avoiding trees and forests for a small number of vertices already results in a strictly larger forest Stanley-Wilf limit. The asymptotics of $f_n$ seem to be quite sensitive to changes in $t_k$ and $f_k$ for small $k$, at least for uncovered sets $S$. This is in sharp contrast with the situation for permutations. For example, there is only one permutation of $[n]$ that avoids $21$, namely $1,\ldots,n$. We then consider permutations of $[n]$ that avoid $\{213,231,312,321\}$, the set of all patterns of length at least $3$ that do not start with $1$. For $n>1$ there are only two such permutations, given by $1\cdots n$ and $1\cdots(n-2)n(n-1)$. The discrepancy between the number of permutations of $[n]$ that avoid $21$ and $\{213,231,312,321\}$ for $n=2$ is not magnified for larger $n$. For forests, however, there are exponentially many more forests on $[n]$ avoiding $\{213,231,312,321\}$ than there are forests on $[n]$ avoiding $21$. This can intuitively be explained by the observation that there are generally many ways to perturb an increasing forest into another forest that still avoids $\{213,231,312,321\}$. Any vertex whose children are all leaves can swap labels with one of its children, and the resulting forest will still avoid $\{213,231,312,321\}$ (see Figure \ref{tettreefig} for an example). In contrast, when one tries to apply this to the increasing path, corresponding to the permutation $1,\ldots,n$, there is only one way to do so which results in the one other permutation avoiding $\{213,231,312,321\}$. Heuristically, discrepancies between $f_k$ for small $k$ manifest close to the leaves of the forest. There are generally relatively many vertices close to leaves, so the discrepancy is magnified into a strictly larger limit.

\begin{figure}[ht]
\centering
\forestset{filled circle/.style={
      circle,
      text width=4pt,
      fill,
    },}
\begin{forest}
for tree={filled circle, inner sep = 0pt, outer sep = 0 pt, s sep = 1 cm}
[, 
    [, fill=red, edge label={node[left]{1}}
        [, edge label={node[left]{2}}
        ]
        [, edge label={node[left]{3}}
        ]
        [, edge label={node[left]{4}}
        ]
    ]
    [, edge label={node[left]{5}}
        [, fill=red, edge label={node[left]{6}}
            [, edge label={node[left]{8}}
            ]
            [, edge label={node[left]{9}}
            ]
        ]
        [, fill=red, edge label={node[left]{7}}
            [, edge label={node[left]{10}}
            ]
        ]
    ]
    [, edge label={node[right]{11}}
        [, edge label={node[right]{12}}
            [, fill=red, edge label={node[left]{13}}
                [, edge label={node[left]{16}}
                ]
                [, edge label={node[left]{17}}
                ]
                [, edge label={node[left]{18}}
                ]
            ]
            [, edge label={node[left]{14}}
                [, fill=red, edge label={node[left]{19}}
                    [, edge label={node[left]{21}}
                    ]
                    [, edge label={node[right]{22}}
                    ]
                ]
            ]
            [, fill=red, edge label={node[right]{15}}
                [, edge label={node[right]{20}}
                ]
            ]
        ]
    ]
]
\end{forest}
    \caption{Any of the labels of the red vertices, which have labels $1,6,7,13,15,19$, of this increasing forest can be swapped with one of its children, and the resulting forest will still avoid the set $\{213,231,312,321\}$.}
    \label{tettreefig}
\end{figure}
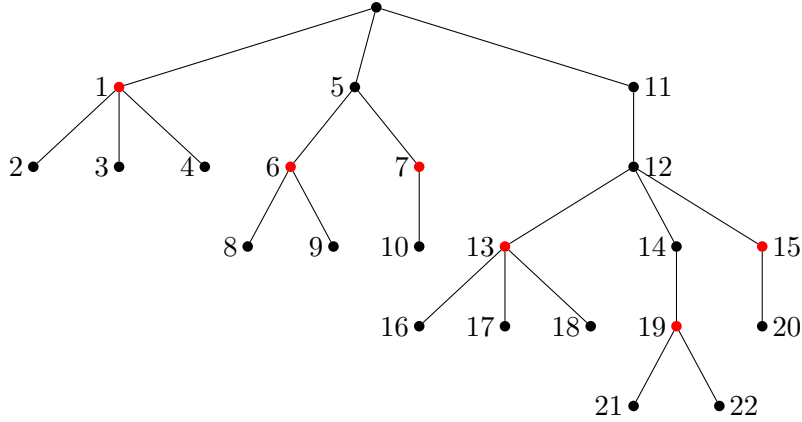

For $S=\{213,231,312,321\}$, we can give more explicit properties of $f_n$ and $t_n$.

\begin{prop}
\label{tplusf}
For $S=\{213,231,312,321\}$, the exponential generating function $T(x)$ of $t_n$ satisfies the differential equation $T'=T+e^T$ with initial condition $T(0)=0$.
\end{prop}

\begin{proof}
By definition, $T(0)=0$, so it suffices to show that $T(x)$ satisfies $T'=T+e^T$. In terms of the coefficients, this reduces to showing the identity $t_{k+1}=f_k+t_k$.

We prove that $t_{k+1}=f_k+t_k$ by casework, depending on where the label $1$ is in a tree on $[k+1]$ avoiding $S$. If $1$ is at the root, then the rest of the tree must be a forest on $\{2,\ldots,k+1\}$ that avoids $S$, and any such forest will work, resulting in $f_k$ such trees. If $1$ is not at the root, then it cannot have any children. We also cannot have the vertex labeled $1$ be at depth more than $2$. Thus, the vertex labeled $1$ must be a child of the root of the tree. Deleting this vertex results in a tree on $\{2,\ldots,k+1\}$, and any such tree can be turned into a tree on $[k+1]$ avoiding $S$ by adding a vertex labeled $1$ as a child of the root. The identity $t_{k+1}=t_k+f_k$ then follows, and the proposition is proven.
\end{proof}

This proposition tells us that the number of trees on $[n]$ avoiding $S$ is $\left((x+e^x)d/dx\right)^nx$ evaluated at $x=0$. By numerically approximating the singularity of the solution to this differential equation, we can obtain an approximation of $L_S$.

\begin{cor}
\label{tplusflim}
The approximation $L_{213,231,312,321}\approx0.4562$ holds.
\end{cor}

This is indeed greater than $e^{-1}\approx0.3679$. Notably, we are able to give an approximation of $L_S$ here instead of just a lower bound because we have an explicit differential equation that $T(x)$ satisfies. Even if the differential equation is not explicitly solveable, we can numerically approximate $L_S$. It seems to be very rare that this is possible, and none of the other uncovered sets $S$ other than the ones containing $21$ seem to satisfy any simple differential equation.

With all of the limits computed so far, one might conjecture that having the same forest Stanley-Wilf limit implies forest-Wilf equivalence. While this may be the case for uncovered sets of patterns, it is not true in general.

\begin{prop}
\label{otherlim1e}
For $S=\{132,231,321\}$, $t_n=n!$, $T(x)=x/(1-x)$, $F(x)=e^{x/(1-x)}$, and $L_{132,231,321}=e^{-1}$.
\end{prop}

\begin{proof}
We first show the following characterization of trees on $[n]$ avoiding $S$. They are the trees that have an arbitrary root label but are otherwise increasing. Indeed, to avoid the patterns in $S$, we cannot have any instances of $21$ not including the root. But as long as no such instances exist, we avoid $S$. There are $n$ ways to select a label for the root and $(n-1)!$ ways to choose the increasing forest underneath the root, for a total of $n!$ ways, as desired.

Consequently, the exponential generating function of $t_n$ is $T(x)=x/(1-x)$. Thus, the exponential generating function of $f_n$ is $F(x)=e^{x/(1-x)}$. The radius of convergence of $T(x)$ and $F(x)$ is $1$, so $\limsup_{n\rightarrow\infty}f_n^{1/n}/n\le e^{-1}$ by Stirling's approximation. But $f_n\ge n!$ as all increasing forests avoid $S$, so $\liminf_{n\rightarrow\infty}f_n^{1/n}/n\ge e^{-1}$, and we obtain the result that $L_{132,231,321}=e^{-1}$.
\end{proof}

Note that Proposition \ref{lim1e} does not apply here because $S$ is not uncovered. We know that the inequality $t_{k+1}\ge f_k$ cannot hold for all $k$, or the same proof for Theorem \ref{fswl} and Proposition \ref{lim1e} would apply. Indeed, $t_9=362880$ while $f_8=394353$. It is not even the case that $f_n=O(t_{n+1})$ here. Vaclav Kotesovec gives the asymptotic growth $f_n\sim1/(\sqrt{2e})n^{n-1/4}e^{2\sqrt n-n}$ on the OEIS for $f_n$ \cite{OEIS}. While we were able to show that the limit exists in this case, the fact that $t_{n+1}=o(f_n)$ for this covered set $S$ suggests that we will not be able to modify our proof of Theorem \ref{fswl} to work in general.

\section{Future Work}
\label{fut}
We conclude this paper by discussing several conjectures, open questions, and potential directions for future research.

\subsection{Asymptotics and forest Stanley-Wilf limits}
\label{asyfut}
\hspace*{\fill} \\
Conjecture \ref{gp72} is still unproven for covered sets of patterns. In the case of uncovered sets $S$, there remains the problem of finding the value of $L_S$ to arbitrary precision, as it does not seem possible in general to find differential equations for the exponential generating functions.

\begin{quest}
\label{upperalg}
Is there an algorithm that computes convergent upper bounds on $L_S$ for uncovered sets $S$?
\end{quest}

Beyond this, we believe that $L_S$ should satisfy certain ``monotonicity'' properties.

\begin{conj}
\label{1isnull}
If $L_S=1$, then $S=\varnothing.$
\end{conj}

\begin{conj}
\label{monolim}
If $\pi$ and $\sigma$ are different patterns such that $\pi$ contains $\sigma$, then $L_\pi>L_\sigma$.
\end{conj}

One possible way to resolve Conjecture \ref{1isnull} is to find an algorithm that answers Question \ref{upperalg} and analyze when the upper bounds it gives are always $1$. Note that Proposition \ref{lim1e} shows Conjecture \ref{monolim} when $\sigma\in\{12,21\}$. The main difficulty in generalizing our proof seems to be obtaining a comparison between $t_{k+1}(\pi)-f_k(\pi)$ and $t_{k+1}(\sigma)-f_k(\sigma)$. All we currently know is that these are nonnegative and equal to $0$ for $12$ and $21$, which is only sufficient to prove the connjecture for $\sigma\in\{12,21\}$.

We also have the following conjecture about sharper asymptotics for $f_n$.

\begin{conj}
\label{sharpasy}
For an uncovered set $S$ of patterns, there exist constants $a_S$ and $b_S$ such that $\frac{f_n}{n!}\sim a_Sn^{b_S}(eL_s)^n$.
\end{conj}

Based on limited data, it seems that $b_S=0$ for nonempty $S$, while by Cayley's formula for $S=\varnothing$, $\frac{f_n}{n!}\sim\frac e{\sqrt{2\pi}}n^{-3/2}e^n$. The case that $S=\varnothing$ seems to be fundamentally different. The asymptotics for covered sets also seem to be very different. For example, for $S=\{1\cdots k,\ell\cdots1\}$, $L_S=0$, but clearly $\frac{f_n}{n!}\not\sim0$. Taking $k=3$ and $\ell=2$, forests avoiding $S$ become increasing forests of depth at most $2$. Such forests are in bijection with partitions of the label set $[n]$, so $f_n$ is given by the $n$th Bell number $B_n$. The asymptotics of $B_n$ are much more complicated than the behavior predicted by Conjecture \ref{sharpasy} for uncovered sets. Yet another example is given by $S=\{132,231,321\}$ from Proposition \ref{otherlim1e}, where $\frac{f_n}{n!}\sim\frac 1{\sqrt{4\pi e}}n^{-3/2}e^{2\sqrt n}$.

Our heuristic for Conjecture \ref{sharpasy} is that for uncovered sets, $F$ is reasonably approximated by series $F_m$ that have a meromorphic continuation to $\mathbb C$. The coefficients of these series all satisfy the type of asymptotic behavior described in the statement of the conjecture, so we believe that $F$ satisfies a similar estimate. This extends to any sets $S$ satisfying $f_n=O(t_{n+1})$ as well.

We in fact predict that the condition $f_n=O(t_{n+1})$ is what distinguishes uncovered sets and covered sets.

\begin{conj}
\label{fot}
A set $S$ of patterns is uncovered if and only if it satisfies $f_n=O(t_{n+1})$.
\end{conj}

Given a forest on $[n]$, there are $n+1$ ways we can extend this to a tree on $[n+1]$. We choose a root label $a$ for the tree in $[n+1]$ and the rest of the tree is the given forest, relabeled with $[n+1]\setminus\{a\}$. The quantity $\frac{t_{n+1}}{f_n}$ can be interpreted as the expected number of root labels we can choose for a uniform random forest on $[n]$ avoiding $S$ such that the resulting tree on $[n+1]$ also avoids $S$. For uncovered $S$, $1$ or $n+1$ is always a valid choice, so this expected value is always at least $1$. We predict that this expected value tends to $0$ for covered sets $S$. Small roots are unlikely to be possible because of the pattern in $S$ starting with $1$, and large roots are unlikely to be possible because of the pattern in $S$ starting with its largest element. While it may be possible that moderately sized roots can keep the expected value high, we conjecture that this is not the case.

One way to find a lower bound for $L_S$ for a covered set $S=\{\pi_1,\ldots,\pi_m\}$ is to consider the limit $L_S'$ for $S'=\{\pi_1',\ldots,\pi_m'\}$, where $\pi_i'$ is a subpattern of $\pi_i$ and $S'$ is an uncovered set. We conjecture that this is also how $L_S$ is achieved, i.e. that there cannot be exponentially more ways to avoid $S$ than there are to avoid $S'$ for the best choice of $S'$.

\begin{conj}
\label{dream}
Define the \emph{reduction} $\widehat\pi$ of a pattern $\pi=\pi(1)\cdots\pi(k)$ to be the pattern of length $k-1$ in the same relative order as $\pi(2)\cdots\pi(k)$. Let $S=\{\pi_1,\ldots,\pi_m\}$ be a covered set of patterns, and let $S_i=S\setminus\{\pi_i\}\cup\{\widehat\pi_i\}$. Then $L_S=\max_{1\le i\le m}L_{S_i}$.
\end{conj}

We can repeatedly replace patterns in $S$ with their reductions until $S$ is an uncovered set, and this yields a lower bound on $L_S$. The conjecture is that $L_S$ is equal to the maximum lower bound achieved in this way. For example, this conjecture predicts that $L_{132,4213}=L_{132,213}$. This suggests a path to proving Conjecture \ref{gp72}. By reducing a pattern in a covered set $S$ of patterns we introduce more forests that contain $S$, and by reducing patterns in $S$ until it is uncovered, we obtain a natural lower bound on the limiting growth rate $L_S$. If we can show that reducing the correct pattern decreases the number of forests that avoid $S$ by a subexponential factor, then the existence of the limit would be shown to be equal to the limit for the uncovered set at the end of the reduction process. Note also that if this conjecture were true, it would provide an answer to the following question.

\begin{quest}
\label{limvals}
What are the possible values of $L_S$?
\end{quest}

The answer would then be the values of $L_S$ over all uncovered sets $S$, which we are able to estimate.

It is possible that Conjecture \ref{dream} is false even in simple cases such as $S=\{132,312\}$. However, it is consistent with Propositions \ref{lim0} and \ref{otherlim1e}. In those cases, the values of $f_n$ exceed the corresponding natural lower bounds by a subexponential factor, on the order of the Bell numbers or $\exp(O(\sqrt n))$. These can be interpreted as a result of more wildly behaved singularities of the exponential generating function $F(x)$ in the neighborhood of $\frac1{eL_S}$. Indeed, for a covered set of patterns, we no longer have the same approximation by meromorphic functions as in the uncovered case, which heuristically suggests more erratic behavior at the singularity.

We make one last generalization of forest Stanley-Wilf limits. Say that a rooted labeled forest $F_1$ \emph{contains} another rooted labeled forest $F_2$ if there exists a graph minor of $F_1$ that is isomorphic to $F_2$ and whose corresponding labels are in the same relative order as $F_2$. For example, the type of pattern avoidance we have been studying in this paper can be viewed as forests avoiding a rooted labeled path. Similar to the closed permutation classes considered in \cite{KK}, we can define a \emph{closed forest class} $\Pi$ to be a collection of rooted labeled forests such that if a forest $F_1\in\Pi$ contains a forest $F_2$, then $F_2\in\Pi$. We can make the following general conjecture about the growth rates of closed forest classes.

\begin{figure}[ht]
    \centering
    \forestset{filled circle/.style={
      circle,
      text width=4pt,
      fill,
    },}
    \begin{minipage}[b]{0.45\linewidth}
    \centering
    \begin{forest}
    for tree={filled circle, inner sep = 0pt, outer sep = 0 pt, s sep = 1 cm}
    [, 
        [, edge label={node[right]{7}}, name=v7
            [, edge label={node[left]{5}}, name=v5
                [, edge label={node[left]{4}}, name=v4
                    [, edge label={node[below]{2}}, name=v2
                    ]
                    [, edge label={node[below]{8}}, name=v8
                    ]
                ]
            ]
            [, edge label={node[right]{6}}, name=v6
                [, edge label={node[right]{3}}, name=v3
                    [, edge label={node[below]{1}}, name=v1
                    ]
                    [, edge label={node[below]{9}}, name=v9
                    ]
                ]
            ]
        ]
    ]
    \path[blue,out=-180,in=135,dashed] (v7.child anchor) edge (v2.parent anchor);
    \path[green,dashed] (v7.child anchor) edge (v8.parent anchor);
    \path[red,dashed] (v7.child anchor) edge (v3.parent anchor);
    \path[cyan,out=-90,in=135,dashed] (v3.child anchor) edge (v9.parent anchor);
    \end{forest}
    \end{minipage}
    \begin{minipage}[b]{0.45\linewidth}
    \centering
    \begin{forest}
    for tree={filled circle, inner sep = 0pt, outer sep = 0 pt, s sep = 1 cm}
    [, fill=white
        [, edge=white, edge label={node[above,black]{3}}
            [, edge=blue, edge label={node[left,black]{1}}
            ]
            [, edge=red, edge label={node[left,black]{2}}
                [, edge=cyan, edge label={node[below,black]{5}}
                ]
            ]
            [, edge=green, edge label={node[right,black]{4}}
            ]
        ]
    ]
    \end{forest}
    \end{minipage}
    \caption{The forest on the left contains the forest pattern on the right. Note that our forests are unordered, so the branches of the pattern can appear in a different order in the forest.}
    \label{forestpatfig}
\end{figure}
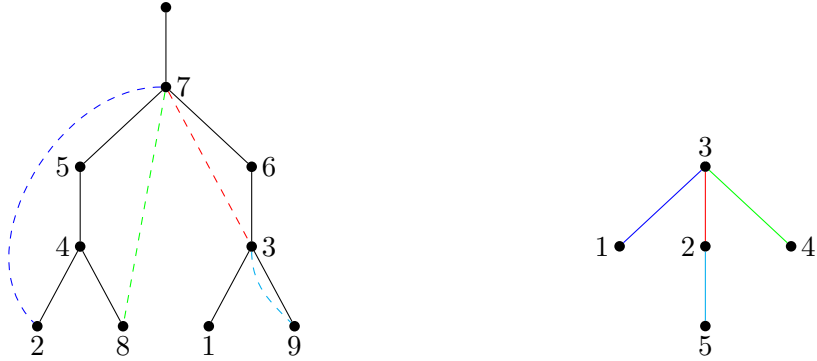

\begin{conj}
\label{superfswl}
Let $\Pi$ be a closed forest class, and let $\Pi_n$ denote the set of forests on $[n]$ in $\Pi$. Then $\lim_{n\rightarrow\infty}\frac{|\Pi_n|^{1/n}}n$ exists.
\end{conj}

Note that forests avoiding a set $S$ of patterns form a closed forest class, so this can be seen as a generalization of Conjecture \ref{gp72}. It would also be interesting to study forests that avoid a set $S$ of \emph{forest patterns} consisting of rooted labeled forests, where the avoidance and containment is in the sense described above for forests. Figure \ref{forestpatfig} gives an example of containment of a forest pattern. This is somewhat reminiscent of the poset pattern avoidance studied by Hopkins and Weiler in \cite{HW}. However, restricting to the setting of forest patterns allows us to carry over results we have shown in this paper. For example, the proof for Theorem \ref{fswl} automatically gives the existence of the forest Stanley-Wilf limit for certain sets of forest patterns.

\subsection{Limiting distributions for forest statistics}
\label{statfut}
\hspace*{\fill} \\
Finally, we make some conjectures about how certain forest statistics are distributed in the limit $n\rightarrow\infty$.

Certain results in permutation pattern avoidance look at how permutation statistics such as the number of inversions and ascents are distributed when we look at permutations avoiding certain patterns rather than the whole symmetric group (see, for example, \cite{E04}). Forests come with their own set of interesting statistics that seem to yield interesting limit distributions. We make a few conjectures about the root of a tree on $[n]$ avoiding $S$ and the number of trees in a forest on $[n]$ avoiding $S$.

For a set $S$ of patterns, let $R_{S,n}$ denote the label of the root of a uniform random tree on $[n]$ avoiding $S$, let $T_{S,n}$ denote the number of trees in a uniform random forest on $[n]$ avoiding $S$, and let $T_{S,n,k}$ denote the number of trees with $k$ vertices in a uniform random forest on $[n]$ avoiding $S$.

\begin{conj}
\label{convdist}
For any set $S$ of patterns, there exists a random variable $R_S$ such that $\frac{R_{S,n}}n$ converges in law to $R_S$ as $n\rightarrow\infty$.
\end{conj}

Note that the limiting distribution can be continuous, such as a uniform distribution when $S=\{132,231,321\}$ by Proposition \ref{otherlim1e}, or discrete, such as a convergence to $0$ when $S=\{21\}$. When $S$ is uncovered, we expect most of the trees to have root labels that are very small or very large. Heuristically, the ``easiest'' way to avoid $S$ when $S$ is uncovered is to have the root have label close to $1$ or $n$. In the case that $S$ contains a pattern starting with $1$, this is no longer true if our root label is $1$, but we can still have a root label close to $n$, and vice versa if $S$ contains a pattern starting with its largest element. We have the following stronger conjecture that formalizes this.

\begin{conj}
\label{berlim}
For any uncovered set $S$ of patterns, $\frac{R_{S,n}}n$ converges in distribution to a Bernoulli random variable $Ber(p)$ for some $p\in[0,1]$. If $S$ contains a pattern starting with $1$, then $p=1$, and if $S$ contains a pattern starting with its largest element, then $p=0$. Furthermore, there exist limiting probabilities $p_1,p_2,\ldots,q_1,q_2,\ldots$ summing to $1$ such that $\mathbb P(R_{S,n}=k)\rightarrow p_k$ and $\mathbb P(R_{S,n}=n+1-k)\rightarrow q_k$ as $n\rightarrow\infty$. If $S$ contains a pattern starting with $1$, then $p_1=p_2=\cdots=0$, and if $S$ contains a pattern starting with its largest element, then $q_1=q_2=\cdots=0$.
\end{conj}

Some data computed for $S=\{213\}$ and $S=\{123\}$ supports this conjecture, but we do not have any data for covered sets $S$.

We now turn to the distribution of $T_{S.n}$ as $n\rightarrow\infty$. Our main motivation comes from the fact that for $S=\{21\}$, i.e. for increasing forests, there exists a bijection between forests on $[n]=P_1\sqcup\cdots\sqcup P_m$ with $m$ components such that the labels in the components are $P_1,\ldots,P_m$ and permutations of $[n]$ with $m$ cycles such that the elements in the cycles are $P_1,\ldots,P_m$. A classical result of Goncharov in \cite{G1, G2} states that in a uniform random permutation $\pi$ of $[n]$, the number of cycles $C_n$ in $\pi$ is asymptotically normal: $\frac{C_n-\mathbb E[C_n]}{\text{Var}(C_n)}$ converges in distribution to a standard Gaussian. Furthermore, $\mathbb E[C_n],\text{Var}(C_n)\sim\log n$. Another result in this area, due to Arratia and Tavaré in \cite{AT}, is that if $C_{n,k}$ is the number of cycles in $\pi$ of length $k$, then $(C_{n,1},C_{n,2},\ldots)$ converges in distribution to $(Z_1,Z_2,\ldots)$ as $n\rightarrow\infty$, where $Z_1,Z_2,\ldots$ are independent Poisson random variables with $\mathbb E[Z_k]=k^{-1}$. The correspondence between trees in increasing forests and cycles in permutations immediately gives us these results but for $T_{S,n}$ instead of $C_n$ for $S=\{21\}$. For example, Goncharov's theorem implies that $\frac{T_{S,n}-\mathbb E[T_{S,n}]}{\text{Var}(T_{S,n})}$ converges in distribution to a standard Gaussian as $n\rightarrow\infty$. We conjecture that these results also hold for other sets of patterns.

\begin{conj}
\label{asynorm}
For all nonempty sets $S$ of patterns, the random variable $T_{S,n}$ is asymptotically normal. In particular, $\frac{T_{S,n}-\mathbb E[T_{S,n}]}{\text{Var}(T_{S,n})}$ converges in distribution to a standard Gaussian as $n\rightarrow\infty$.

Furthermore, if $S$ is uncovered, then $\mathbb E[T_{S,n}],\text{Var}(T_{S,n})=\Theta(\log n)$, and $(T_{S,n,1},T_{S,n,2},\ldots)$ converges in distribution to $(Z_1,Z_2,\ldots)$, where $Z_1,Z_2,\ldots$ are independent Poisson random variables with $\mathbb E[Z_k]=\Theta(k^{-1})$.
\end{conj}

Note that the hypothesis on $S$ being nonempty is necessary. When $S=\varnothing$, the total number of rooted forests on $[n]$ is $(n+1)^{n-1}$ and the total number of rooted trees on $[n]$ is $n^{n-1}$, so $\mathbb P(T_{S,n}=1)\rightarrow\frac1e$ as $n\rightarrow\infty$ and $T_{S,n}$ cannot be asymptotically normal in this case. The behavior of the limiting distribution is related to the behavior of the exponential generating function $F(x)$ around its singularity $\frac1{eL_S}$ by \cite[Section IX.4]{FS}. Indeed, as mentioned previously, for uncovered $S$ we expect $F(x)$ to be well-behaved because of the approximation by $F_m(x)$, which has a meromorphic continuation to $\mathbb C$. However, this shows a shortcoming of our method, which does not distinguish between when $S$ is empty and when $S$ is nonempty. More sophisticated analysis of the singularity of $F(x)$ is needed if we are to prove the conjecture using this approach. Data computed for all of the uncovered sets we considered in this section supports this conjecture. On the other hand, $T_{S,n}$ does seem to be asymptotically normal for covered sets $S$ as well, but the point of concentration is different, most likely due to the different behavior of $F(x)$ around its singularity. In the case of $S=\{132,231,321\}$, it appears that $\mathbb E[T_{S,n}]\sim\sqrt n$. The case of $S=\{123,21\}$ is equivalent to the distribution of Stirling numbers of the second kind. This problem was considered by Harper in \cite{H}, and Harper's result translates to the asymptotic normality of $T_{S,n}$. The mean, however, is of a different order than $\sqrt n$ and $\log n$. It appears that a variety of asymptotics can occur for the mean of $T_{S,n}$ for covered sets, in contrast to uncovered sets.

It would also be interesting to examine other forest statistics as well. Some that we did not consider include the depth of the forest, the number of leaves in the forest, and the degree of the root of a random tree in the forest.

\section*{Acknowledgments} This research was funded by NSF-DMS grant 1949884 and NSA grant H98230-20-1-0009. The author thanks Amanda Burcroff, Swapnil Garg, and Alan Peng for fruitful discussions about this research and for reading drafts of the paper, as well as Noah Kravitz, Ashwin Sah, and Fan Zhou for helpful suggestions. The author also thanks Professor Joe Gallian for suggesting the problem and running the Duluth REU in which this research was conducted and Amanda Burcroff, Colin Defant, and Yelena Mandelshtam for fostering a productive virtual research environment through their mentorship. Finally, the author thanks the anonymous referees for their careful reading and insightful comments that improved the presentation of this paper.

\textsc{University of Cambridge, The Old Schools, Trinity Ln, Cambridge CB2 1TN}

\emph{E-mail address: }\href{mailto:mr918@cam.ac.uk}{\tt mr918@cam.ac.uk}
\end{document}